\documentclass[a4paper,reqno]{amsart}
\usepackage{graphicx,color,ae,fancyvrb}
\usepackage[T1]{fontenc}
\usepackage[authoryear,round]{natbib}

\usepackage{amsmath,amssymb,enumerate}
\usepackage{amsthm}
\usepackage{multirow,array,booktabs}
\DeclareMathOperator{\corr}{corr}
\DeclareMathOperator{\ME}{\mathsf{E}}
\DeclareMathOperator{\indicator}{\mathsf{1}}
\DeclareMathOperator{\constone}{\indicator_{[0,T]}}
\DeclareMathOperator{\stepzerotot}{\indicator_{[0,t]}}
\DeclareMathOperator{\cov}{cov}
\DeclareMathOperator{\var}{var}

\DeclareMathOperator{\Betafunction}{B}

\newcommand*{\norm}[1]{\left\lVert#1\right\rVert}
\newcommand*{\abs}[1]{\left\lvert#1\right\rvert}
\newtheorem{theorem}{Theorem}[section]
\newtheorem{lemma}[theorem]{Lemma}
\newtheorem{prop}[theorem]{Proposition}
\newtheorem{corollary}[theorem]{Corollary}
\theoremstyle{definition}
\newtheorem{assumption}{Assumption}
\newtheorem{example}[theorem]{Example}
\theoremstyle{remark}
\newtheorem{remark}[theorem]{Remark}
\newcommand{\otoprule}{\midrule[\heavyrulewidth]}
\newcolumntype{Q}{>{$}l<{$}}

\providecommand{\keywords}[1]{\textit{Keywords:} #1}
\newcommand{\Prob}{\mathsf{P}}

\begin{document}
\author{Yuliya Mishura}
\address{Department of Probability Theory, Statistics and Actuarial Mathematics, Taras Shev\-chen\-ko National University of Kyiv,
64 Volodymyrska, 01601 Kyiv, Ukraine} \email{myus@univ.kiev.ua}

\author{Kostiantyn Ralchenko}
\address{Department of Probability Theory, Statistics and Actuarial Mathematics, Taras Shev\-chen\-ko National University of Kyiv,
64 Volodymyrska, 01601 Kyiv, Ukraine}
\email{k.ralchenko@gmail.com}

\author{Sergiy Shklyar}
\address{Department of Probability Theory, Statistics and Actuarial Mathematics, Taras Shev\-chen\-ko National University of Kyiv,
64 Volodymyrska, 01601 Kyiv, Ukraine}
\email{shklyar@univ.kiev.ua}

\title[Maximum likelihood drift estimation for Gaussian process]{Maximum likelihood drift estimation\\ for Gaussian process with stationary increments}

\begin{abstract}
The paper deals with the regression model
$X_t = \theta t + B_t$,\linebreak $t\in[0, T ]$,
where $B=\{B_t, t\geq 0\}$ is a centered Gaussian process with stationary increments.
We study the estimation of the unknown parameter $\theta$ and establish the formula for the likelihood function in terms of a solution to an integral equation.
Then we find the maximum likelihood estimator and prove its strong consistency. The results obtained generalize the known results for fractional and mixed fractional Brownian motion.
\end{abstract}

\maketitle

\section{Introduction}

We study the problem of the drift parameter estimation for the stochastic process
\begin{equation}\label{eq:proc}
X_t = \theta t + B_t,
\end{equation}
where $\theta \in \mathbb{R}$ is an unknown parameter, and $B=\{B_t, t\geq 0\}$ is a centered Gaussian process with stationary increments,
$B_0 = 0$.
In the particular case when $B=B^H$ is a fractional Brownian motion, this model has been studied by many authors.
Mention the paper \cite{Norros1999} that  treats the maximum likelihood estimation by continuous observations of the trajectory of $X$ on the interval $[0,T]$ (see also \cite{Breton98}).
Further, the paper \cite{HNXZ11} investigates the exact maximum likelihood estimator by discrete observations at the points $tk=kh$, $k=1,2,\ldots,N$; the paper
\cite{BTT} considers the maximum likelihood estimation in the discrete scheme of observations, where the trajectory of $X$ is observed at the points
$t_k=\frac{k}{N}$, $k=1,2,\ldots,N^\alpha$, $\alpha>1$.
For hypothesis testing of the drift parameter sign in the model \eqref{eq:proc} driven by a fractional Brownian motion, see \cite{Stiburek16}.
The paper  \cite{CaiChigKlept} treats  the  likelihood function   for Gaussian processes not necessarily having stationary increments. However, on the one hand, our approach is different from their one, it cannot be deduced from their general formulas and on the other hand, gives  rather elegant representations.
The construction of the maximum likelihood estimator in the case when $B$ is the sum of two fractional Brownian motions was studied in \cite{MiSumfbms1} and \cite{MishuraVoronov}.
A similar non-Gaussian model driven by the Rosenblatt process was considered in \cite{BTT}.

As already mentioned,  we consider the case when $B$ is a centered Gaussian processes with stationary increments. We construct the maximum likelihood estimators for both discrete and continuous schemes of observations.
The assumptions on the process in the continuous case are formulated in terms of the second derivative of its covariance function, see Assumptions~\ref{assump-B} and \ref{assump-hexists}.
The exact formula for the maximum likelihood estimator
contains a solution of an integral equation with  the  kernel obtained after the differentiation.
We give the sufficient conditions for the strong consistency of the estimators. Several examples of the process $B$ are considered.

The paper is organized as follows.
Section~\ref{sec:discr} is devoted to the case of the discrete observations. The maximum likelihood estimation for continuous time is studied in Section~\ref{sec:cont}.

\section{Maximum likelihood estimation by discrete observations}\label{sec:discr}
We start with the construction of the likelihood function and the maximum likelihood estimator in the case of discrete observations. In the next section these results will be used for the
derivation of the likelihood function in the continuous-time case, see the proof of Theorem~\ref{th-L}.

Let the process $X$   defined by formula  \eqref{eq:proc} be observed at the points
$t_k$,\linebreak $k=0,1,\ldots,N$,
\[
0 = t_0 < t_1 < \ldots < t_N  = T .
\]
The problem is to estimate the parameter $\theta$ by the observations $X_{t_k}$, $k=0, 1, \ldots, N$ of the process $X_t$.

\subsection{Likelihood function and construction of the estimator}

Denote
\[
\Delta X^{(N)} = \left(X_{t_k} - X_{t_{k-1}}\right)_{k=1}^N, \quad
\Delta B^{(N)} = \left(B_{t_k} - B_{t_{k-1}}\right)_{k=1}^N.
\]
Note that in our model $X_{t_0} = X_0 = 0$, and the $N$-dimensional vector $\Delta X^{(N)}$ is a one-to-one function of the observations.
The vectors $\Delta B^{(N)}$ and $\Delta X^{(N)}$ are Gaussian with different means (except the case $\theta=0$) and the same covariance matrix.
We denote this covariance matrix by $\Gamma^{(N)}$.
The next maximum likelihood estimator coincides with the least square estimator considered in \citet[eq.~(4a.1.5)]{Rao:2002}.
\begin{lemma}\label{pr:L^N}
Assume that the Gaussian distribution of the vector
$(B_{t_k})_{k=1}^N$ is nonsingular.
Then one can take the function
\begin{equation}
\label {l-def-114}
L^{(N)}_{\Delta X^{(N)} = x}(\theta)
=  \frac{f_{\theta} (x)}{f_{0} (x)}
= \exp\left\{
\theta z^\top \bigl(\Gamma^{(N)}\bigr)^{-1} x -
\frac{\theta^2}{2} z^\top \bigl(\Gamma^{(N)}\bigr)^{-1} z
\right\} ,
\end{equation}
where $ z = (t_k - t_{k-1})_{k=1}^N$,
as a likelihood function in the discrete-time model.
MLE is linear with respect to the observations and equals
\begin{equation}\label{eqMLEdisc}
\hat\theta^{(N)} = \frac{z^\top \left(\Gamma^{(N)}\right)^{-1} \Delta X^{(N)}}
{z^\top \left(\Gamma^{(N)}\right)^{-1} z} .
\end{equation}
\end{lemma}
\begin{proof}
The pdf of $\Delta B^{(N)}$ with respect to the Lebesgue measure equals
\[
f_{\theta} (x)=
\frac{1}{(2 \pi)^{N/2} \sqrt{\det \Gamma^{(N)}}}
\exp\left\{
- \frac{1}{2} ( x - \theta z )^\top  \bigl(\Gamma^{(N)}\bigr)^{-1}
( x - \theta z ) \right\} .
\]
The density of the observations for given $\theta$ with respect to the distribution of the observations for $\theta=0$
is taken as a likelihood function.
\end{proof}

\begin{remark}\label{rem-Toeplitz}
Let the process $X$ be observed on a regular grid,
i.e., at the points $t_k = k h$,
$k=1, \ldots, N$, where $h>0$.
Then $\Gamma^{(N)}$ is a Toeplitz matrix, that is
\begin{align*}
\Gamma^{(N)}_{k+l, l} =
\Gamma^{(N)}_{l, k+l} & =
\ME\left(B_{(k+l)h} - B_{(k+l-1)h}\right) \left(B_{l h} - B_{(l-1)h}\right)\\ &=
\ME B_{(k+1)h} B_{h} - \ME B_{k h} B_{h}
\end{align*}
does not depend on $l$ due to the stationarity of increments.
This simplifies the numerical computation of MLE.
\end{remark}

\subsection{Properties of the  estimator}
Since $\Delta X^{(N)} = \Delta B^{(N)} + \theta z$,
the maximum likelihood estimator \eqref{eqMLEdisc}
equals
\[
\hat\theta^{(N)} = \theta +
\frac{z^\top \left(\Gamma^{(N)}\right)^{-1} \Delta B^{(N)}}
{z^\top \left(\Gamma^{(N)}\right)^{-1} z}\, .
\]

\begin{lemma}\label{discrprop}
Under assumptions of Lemma~\ref{pr:L^N}, the estimator $\hat\theta^{(N)}$
is unbiased and  normally distributed. Its variance equals
\[
\var \hat\theta^{(N)} = \frac{1}{z^\top \left(\Gamma^{(N)}\right)^{-1} z} \,.
\]
\end{lemma}
\begin{proof}
The estimator $\hat\theta^{(N)}$ is unbiased and normally distributed
because $\hat\theta^{(N)} - \theta$ is linear and  centered Gaussian
vector $\Delta B^{(N)}$.
The variance of the estimator is equal to
\begin{align*}
\var \hat\theta^{(N)} &=
\frac{\var\left( z^\top \left(\Gamma^{(N)}\right)^{-1} \Delta B^{(N)} \right)}
{\left(z^\top \left(\Gamma^{(N)}\right)^{-1} z\right)^2}
=\frac{z^\top \left(\Gamma^{(N)}\right)^{-1} \var\left(\Delta B^{(N)}\right) \left(\Gamma^{(N)}\right)^{-1} z}
{\left(z^\top \left(\Gamma^{(N)}\right)^{-1} z\right)^2} \\ &=
\frac{z^\top \left(\Gamma^{(N)}\right)^{-1} \Gamma^{(N)} \left(\Gamma^{(N)}\right)^{-1} z}
{\left(z^\top \left(\Gamma^{(N)}\right)^{-1} z\right)^2} =
\frac{1}
{z^\top \left(\Gamma^{(N)}\right)^{-1} z}\,.
\qedhere
\end{align*}
\end{proof}

To prove the consistency of the estimator, we need the following technical result.
\begin{lemma}
If $A \in \mathbb{R}^{N \times N}$ is a positive definite matrix,
$x \in \mathbb{R}^N$, $x\neq 0$ is a non-zero vector, then
\[
 x^\top A^{-1} x \ge \frac {\|x\|^4} {x^\top A x}\, .
\]
\end{lemma}
\begin{proof}
As the matrix $A$ is positive definite, $x^\top A x > 0$
and there exists positive definite matrix $A^{1/2}$
(and so the matrix $A^{1/2}$ is symmetric and nonsingular)
such that $(A^{1/2})^2 = A$.
By the Cauchy-Schwarz inequality,
\[
\norm{x}^4 = \left(x^\top A^{1/2} A^{-1/2} x\right)^2 \le
\norm{A^{1/2} x}^2 \norm{A^{-1/2} x}^2 = \left(x^\top A x\right) \left(x^\top A^{-1} x\right),
\]
whence the desired inequality follows.
\end{proof}

In the rest of this section we assume that the process $X$ is observed on a regular grid, at the points $t_k=kh$, $k=1,\ldots,N$, for some $h>0$. We also assume that for any $N$ the Gaussian distribution of the vector $(B_{kh})_{k=1}^N$ is nonsingular.
\begin{theorem}\label{thm-consdiskr}
Let $h > 0$,
and
\[
 \ME \left(B_{(k+1) h} - B_{k h}\right) B_{h} \to 0 \quad
 \textrm{as} \quad
 k\to\infty.
\]
Let $\hat\theta^{(N)}$ be the ML estimator of parameter
$\theta$ of the model \eqref{eq:proc} by the observations
$X_{k h}$, $k=1,\ldots, N$.
Then the estimator $\hat\theta^{(N)}$ is mean-square consistent,
i.e.,
\[
\ME \left(\hat\theta^{(N)} - \theta\right)^2 \to 0
\quad \textrm{as} \quad N\to\infty .
\]
\end{theorem}

\begin{proof}
By Remark~\ref{rem-Toeplitz}, the matrix $\Gamma^{(N)}$
has Toeplitz structure,
\[
\Gamma^{(N)}_{l,m} = \Gamma^{(N)}_{l,m}
=
\ME \left(B_{(|l-m|+1) h} - B_{|l-m| \, h}\right) B_{h} .
\]
Moreover, $\Gamma^{(N)}_{k,l}$ does not depend on $N$ as soon as
$N \ge \max(k, l)$.
By Toeplitz theorem,
\begin{align*}
\frac{1}{N^2} \sum_{l=1}^N \sum_{m=1}^N \Gamma_{l,m}^{(N)}
&=
\frac{1}{N} \ME (B_{h})^2 - \sum_{k=2}^N
\frac{2 (N+1-k)}{N^2}
\ME \left(B_{k h} - B_{(k-1) h}\right) B_{h}
\to \\ &\to \lim_{k\to\infty}
\ME \left(B_{k h} - B_{(k-1) h}\right) B_{h} = 0
\quad \textrm{as} \quad
N\to\infty.
\end{align*}

For regular grid we have that  $z = (h, \ldots, h)^\top$. Hence, in this case,
\[
z^\top \Gamma^{(N)} z = h^2 \sum_{l=1}^N \sum_{m=1}^N
\Gamma_{l,m}^{(N)},
\qquad \|z\| = h \sqrt{N}.
\]

Finally, with use of Lemma~\ref{discrprop},
\begin{align*}
\ME\left(\hat\theta^{(N)} - \theta\right)^2 &=
\frac{1}{z^{\top} \left(\Gamma^{(N)}\right)^{-1} z} \le
\frac{z^{\top}\Gamma^{(N)} z}{\|z\|^4} =
\frac{1}{h^2 N^2} \sum_{l=1}^N \sum_{m=1}^N
\Gamma_{l,m}^{(N)} \to 0,
\end{align*}
 as $N\to\infty$.
\end{proof}

To prove the strong consistency, we need the following auxiliary statement.
\begin{lemma}\label{lem-indincrdiscr}
Let $h>0$, and $\hat\theta^{(N)}$ be the ML estimator of parameter $\theta$ of the model \eqref{eq:proc} by the observations $X_{kh}$, $k=1,\ldots,N$.
Then the random process $\hat\theta^{(N)}$ has independent increments.
\end{lemma}
\begin{proof}
In the next paragraph, $N_2 \le N_3$ are positive integers, $$(I, 0) = (I_{N_2}, 0_{N_2\times(N_3-N_2)})$$
is $N_2 \times N_3$ diagonal matrix with ones on the diagonal, and its transpose is denoted by
$\left( \begin{smallmatrix} I \\ 0 \end{smallmatrix} \right)$.
The vector
$\Delta B^{(N_2)} = (B_{h},\ldots,B_{(N_2-1)h}-B_{N_2 h})^\top$
is the beginning of vector
$\Delta B^{(N_3)} = (B_{h},\ldots,B_{(N_3-1)h}-B_{N_3 h})^\top$;
the vector $z_{N_2} = (h, \ldots, h)^\top \in \mathbb{R}^{N_2}$
is the beginning of vector $z_{N_3}=(h, \ldots, h)^\top \in \mathbb{R}^{N_3}$,
so
\[
\Delta B^{(N_2)} = (I,0) \, \Delta B^{(N_3)}, \qquad
z_{N_2} = (I,0) \, z_{N_3}.
\]
Then
\begin{gather*}
\ME \Delta B^{(N_3)} \bigl(\Delta B^{(N_2)}\bigr)^\top = \ME \Delta
N^{(N_3)} \bigl(\Delta B^{(N_3)}\bigr)^\top \left(\begin{smallmatrix} I \\ 0
\end{smallmatrix}\right) =
\Gamma^{(N_3)} \left(\begin{smallmatrix} I \\ 0 \end{smallmatrix}\right),\\
\begin{aligned}
\ME \hat\theta^{(N_3)} \hat\theta^{(N_2)}
&=
\frac{\ME z_{N_3}^\top \left(\Gamma^{(N_3)}\right)^{-1} \Delta B^{(N_3)}
\left(\Delta B^{(N_2)}\right)^\top \left(\Gamma^{(N_2)}\right)^{-1} z_{N_2}^{}}
{z_{N_3}^\top \left(\Gamma^{(N_3)}\right)^{-1} z_{N_3}^{} \,
 z_{N_2}^\top \left(\Gamma^{(N_2)}\right)^{-1} z_{N_2}^{}}
= \\ & =
\frac{z_{N_3}^\top \left(\Gamma^{(N_3)}\right)^{-1} \Gamma^{(N_3)}
   \left(\begin{smallmatrix} I \\ 0 \end{smallmatrix}\right)
   \left(\Gamma^{(N_2)}\right)^{-1} z_{N_2}^{}}
{z_{N_3}^\top \left(\Gamma^{(N_3)}\right)^{-1} z_{N_3}^{} \,
 z_{N_2}^\top \left(\Gamma^{(N_2)}\right)^{-1} z_{N_2}^{}}
= \\ & =
\frac{z_{N_3}^\top
   \left(\begin{smallmatrix} I \\ 0 \end{smallmatrix}\right)
   \left(\Gamma^{(N_2)}\right)^{-1} z_{N_2}^{}}
{z_{N_3}^\top \left(\Gamma^{(N_3)}\right)^{-1} z_{N_3}^{} \,
 z_{N_2}^\top \left(\Gamma^{(N_2)}\right)^{-1} z_{N_2}^{}}
= \\ & =
\frac{z_{N_2}^\top
  \left (\Gamma^{(N_2)}\right)^{-1} z_{N_2}^{}}
{z_{N_3}^\top\left(\Gamma^{(N_3)}\right)^{-1} z_{N_3}^{} \,
 z_{N_2}^\top \left(\Gamma^{(N_2)}\right)^{-1} z_{N_2}^{}}
=
\frac{1}
{z_{N_3}^\top \left(\Gamma^{(N_3)}\right)^{-1} z_{N_3}^{}}.
\end{aligned}
\end{gather*}
For $N_1 \le N_2 \le N_3$
\begin{align*}
\ME \hat\theta^{(N_3)}
    \left(\hat\theta^{(N_2)} - \hat\theta^{(N_1)}\right)
&= \ME \hat\theta^{(N_3)} \hat\theta^{(N_2)} - \ME \hat\theta^{(N_3)} \hat\theta^{(N_1)}\\
&= \frac{1}{z_{N_3}^\top \left(\Gamma^{(N_3)}\right)^{-1} z_{N_3}^{}}
- \frac{1}{z_{N_3}^\top \left(\Gamma^{(N_3)}\right)^{-1} z_{N_3}^{}} = 0;
\end{align*}
therefore for $N_1 \le N_2 \le N_3 \le N_4$
\[
\ME \left(\hat\theta^{(N_4)} - \hat\theta^{(N_3)}\right)
    \left(\hat\theta^{(N_2)} - \hat\theta^{(N_1)}\right) = 0 .
\]
Thus, the Gaussian process $\{\hat\theta^{(N)}, \; N=1,2,\ldots\}$ is proved to have uncorrelated
increments.  Hence its increments are independent.
\end{proof}

\begin{theorem}
Under the assumptions of Theorem~\ref{thm-consdiskr},
the estimator $\theta^{(N)}$ is strongly consistent,
i.e. $\hat\theta^{(N)} \to \theta$ as $N\to\infty$ almost surely.
\end{theorem}

\begin{proof}
By Theorem~\ref{thm-consdiskr} $\var\hat\theta^{(N)} \to 0$ as
$N\to\infty$, so
\begin{align*}
 \var\left(\hat\theta^{(N)} - \hat\theta^{(N_0)}\right) &=
 \var\hat\theta^{(N)} + \var\hat\theta^{(N_0)} - 2 \sqrt{\var\hat\theta^{(N)}  \var\hat\theta^{(N_0)}} \corr\left(\hat\theta^{(N)}, \hat\theta^{(N_0)}\right)\\
 &\to \var\hat\theta^{(N_0)}
\end{align*}
as $N\to\infty$.
The process $\hat\theta^{(N)}$ has independent increments.
Therefore by Kolmogorov's inequality, for $\epsilon>0$ and $N \in\mathbb{N}$
\[
\Prob\left(\sup_{N\ge N_0}
\abs{\hat\theta^{(N)} - \hat\theta^{(N_0)}} > \frac{\epsilon}{2}\right)
\le \frac{4}{\epsilon^2} \lim_{N\to\infty}
\var\left(\hat\theta^{(N)} - \hat\theta^{(N_0)}\right)
= \frac{4}{\epsilon^2} \var \hat\theta^{(N_0)} .
\]
Then, using the unbiasedness of the estimator, we get
\begin{align*}
\Prob\left(
\sup_{N\ge N_0} \abs{\hat\theta^{(N)} - \theta} \ge \epsilon \right)
&\le
\Prob\left(\abs{\hat\theta^{(N_0)} - \theta} \ge \frac{\epsilon}{2} \right)
+
\Prob\left(
\sup_{N\ge N_0} \abs{\hat\theta^{(N)} - \hat\theta^{(N_0)}}
 \ge \frac{\epsilon}{2} \right)
\le \\ & \le
\frac{4}{\epsilon^2} \var\hat\theta^{(N_0)} + \frac{4}{\epsilon^2} \var\hat\theta^{(N_0)} =
\frac{8}{\epsilon^2} \var\hat\theta^{(N_0)} \to 0,
\end{align*}
as $N_0\to\infty$, whence $\abs{\hat\theta^{(N)} - \theta} \to 0$ as $N \to \infty$ almost surely.
\end{proof}

\begin{example}
Let us consider the model~\eqref{eq:proc} with  $B_t=B^{H_1}_t+B^{H_2}_t$, where $B^{H_1}_t$ and $B^{H_2}_t$ are two independent fractional Brownian motions with Hurst indices $H_1,H_2\in(0,1)$, i.\,e.
centered Gaussian processes with covariance functions
\[
\ME B^{H_i}_tB^{H_i}_s
=\frac12\left(t^{2H_i}+s^{2H_i}-\abs{t-s}^{2H_i}\right),
\quad t\ge0,\: s\ge0,\: i=1,2.
\]
These processes have stationary increments, and
\[
\ME\left(B^{H_i}_{(k+1) h} - B^{H_i}_{k h}\right) B^{H_i}_{h}
\sim h^{2H_i}H_i\left(2H_i-1\right)k^{2H_i-2}\to0,
\quad\text{as } k\to\infty,
\]
see e.\,g.\ \cite[Sec. 1.2]{mishura}.
Taking into account the independence of centered processes
$B^{H_1}_t$ and $B^{H_2}_t$,
we obtain that
\[
 \ME (B_{(k+1) h} - B_{k h}) B_{h}
 =\ME\left(B^{H_1}_{(k+1) h} - B^{H_1}_{k h}\right) B^{H_1}_{h}
 +\ME\left(B^{H_2}_{(k+1) h} - B^{H_2}_{k h}\right) B^{H_2}_{h}
 \to 0,
\]
as $k\to\infty$.
Thus, the assumptions of Theorem \ref{thm-consdiskr} are satisfied.
\end{example}

\section{Maximum likelihood estimation by continuous observations}\label{sec:cont}
Let the process $X$ be observed on the whole interval $[0, T]$.
It is required to estimate the unknown parameter $\theta$ by these observations.

\subsection{Likelihood function and construction of the estimator}

In this section we construct a formula for continuous-time MLE, similar to the formula \eqref{eqMLEdisc} for the discrete case.

\begin{assumption}\label{assump-B}
The covariance function of $B_t$ has a mixed derivative
\[
\frac{\partial^2}{\partial s\, \partial t}
\left( \ME B_t B_s \right) = K(t-s),
\]
where $K(t)$ is an even function, $K \in L_1[-T, T]$.
\end{assumption}

\begin{lemma}\label{lem-covfun41}
Under Assumption~\ref{assump-B}, the integral
$\int_0^T f(t) \, dB_t$ exists as the mean square limit of the corresponding Riemann sums for any $f \in L_2[0,T]$.
Moreover,
\begin{equation}\label{eq-cov-int}
\ME \left[
\int_0^T f(t) \, dB_t \: \int_0^T g(s) \, dB_s \right] =
\int_0^T f(t) \int_0^T K(t-s) g(s) \, ds \, dt
\end{equation}
for any $f, g \in L_2[0,T]$.
\end{lemma}
\begin{proof}
According to \cite{HuangCambanis}, the integral
$\int_0^T f(t) \, dB_t$ exists if and only if the double Riemann integral
$\int_0^T\!\!\int_0^T f(t) f(s) K(t-s) \, ds \, dt$
exists. Moreover, if the both integrals $\int_0^T f(t) \, dB_t$ and $\int_0^T g(s) \, dB_s$ exist, then  the formula \eqref{eq-cov-int} holds.
However, using the properties of a convolution, one can prove that
\[
\int_0^T\!\!\int_0^T f(t) f(s) K(t-s) \, ds \, dt
\le \norm{K}_{L_1[-T,T]}\,\norm{f}_{L_2[0,T]}^2<\infty.
\qedhere
\]
\end{proof}

Define a linear operator
$\Gamma_T\colon L_2[0,T] \to L_2[0,T]$ by
\begin{equation}\label{defGammaK}
\Gamma_T f (t) = \int_0^T K(t-s) f(s) \,ds.
\end{equation}
It follows from~\eqref{defGammaK} that
\begin{equation}
\label{eq5a}
\ME \left[
\int_0^T f(t) \, dB_t \: \int_0^T g(s) \, dB_s \right] =
\int_0^T \Gamma_T f (t)  g(t) \, dt.
\end{equation}

The basic properties of the operator $\Gamma_T$ are collected in the following evident lemma.
\begin{lemma}\label{lem-operatorGamma}
Let Assumption~\ref{assump-B} hold. Then
\begin{enumerate}[(i)]
\item
The operator $\Gamma_T$ is bounded ($\|\Gamma_T\| \le \|K\|_{L_1[-T,T]}$)
and self-adjoint;
\item
The following relation between the operator $\Gamma_T$ and the covariance matrix $\Gamma^{(N)}$ from Proposition~\ref{pr:L^N} holds:
\[
M \Gamma_T M^*  = \Gamma^{(N)},
\]
(i.\,e.\ the matrix of the operator $M \Gamma_T M^*\colon \mathbb{R}^N \to \mathbb{R}^N$
equals $\Gamma^{(N)}$),
where $M\colon L_2[0,T] \to \mathbb{R}^N$ and $M^*\colon \mathbb{R}^N \to L_2[0,T]$
are mutually adjoint linear operators,
\begin{gather*}
M f = \left( \int_{t_{k-1}}^{t_k} f(s) \, ds \right) _{k=1}^N, \qquad
M^* x = \sum_{k=1}^n {x_k} \indicator_{[t_{k-1}, t_k]}.
\end{gather*}
\end{enumerate}
\end{lemma}

Now we are ready to formulate our key assumption on the kernel $K$ (in terms of the operator~$\Gamma_T$).
\begin{assumption}\label{assump-hexists}
For all $T>0$, the constant function $\constone(t) = 1$, $t\in[0,T]$,
belongs to the range of the operator $\Gamma_T$, i.\,e.\ there exists a function $h_T \in L_2[0,T]$ such that
\[
\Gamma_T h_T = \constone .
\]
\end{assumption}


\begin{theorem}\label{th-L}
If all finite-dimensional distributions of the process $\{B_t,\allowbreak \; t\in(0,T]\}$,
are nonsingular and
Assumptions~\ref{assump-B} and \ref{assump-hexists}
hold, then
\begin{equation}\label{eq-L}
L(\theta) = \exp\left\{
\theta \int_0^T h_T(s) \, dB_s - \frac{\theta^2}{2}
\int_0^T h_T(s) \, ds\right\}
\end{equation}
is a likelihood function.
\end{theorem}

\begin{proof}
Let us show that the function $L(\theta)$ defined in \eqref{eq-L}
is a density function of a distribution
of the process $X_{t}$ for given $\theta$ with respect to the density function of a distribution of the process $B_t$
(it coincides with $X_t$ when $\theta=0$).
In other words, we need to prove that
\[
d P_\theta = L(\theta) \, d P_0,
\]
where $P_\theta$ is the probability measure that corresponds to the value of the parameter $\theta$.
It suffices to show that for all partitions
$0 =t_0 < t_1 < \ldots < t_N \le T$
of the interval $[0, T]$ and for all cylinder sets $A \in \mathcal{F}_N$ the following equality holds:
\begin{equation}
\label{eq-152-desired}
\int_A d P_\theta =  \int_A L(\theta) \, d P_0 ,
\end{equation}
where $\mathcal{F}_N$ is the $\sigma$-algebra, generated by the values $B_{t_k}$ of the process $B_t$ at the points $t_k$, $k=1,\ldots,N$.
We have
\begin{gather*}
\int_A d P_\theta = \int_A L^{(N)}(\theta) \, P_0, \\
\int_A L(\theta) \, d P_0 = \int_A \ME\nolimits_{\theta=0}[ L(\theta) \mid \mathcal{F}_N] \, d P_0,
\end{gather*}
where $L^{(N)}$ is the likelihood function
\eqref{l-def-114} for the discrete-time model.
To prove \eqref{eq-152-desired},
it suffices to show that
\[
L^{(N)}(\theta) = \ME\nolimits_{\theta=0}[ L(\theta) \mid \mathcal{F}_N].
\]

If $\theta=0$, then $X_t = B_t$,
\[
\ME\nolimits_{\theta=0} [ L(\theta) \mid \mathcal{F}_N ] =
\ME \exp \left\{
\theta \int_0^T h_T(s) \, dB_s
- \frac{\theta^2}{2} \int_0^T h_T(s) ds \right\}.
\]

Due to joint normality of $\int_0^T h_T(s) \, dB_s$ and $\Delta B^{(N)}$,
the conditional distribution of $\int_0^T h_T(s) \, dB_s$
with respect to $\mathcal{F}_N$ is Gaussian
\citep[Theorem 2.5.1]{Anderson:2003};
its conditional variance is nonrandom.
Let us find its parameters.
By the least squares method,
\[
\ME [B_t \mid \mathcal{F}_N] =
\ME \left[B_t \mid \Delta B^{(N)} \right]
= \cov(B_t, \Delta B^{(N)})
\left(\cov(\Delta B^{(N)}, \Delta B^{(N)})\right)^{-1} \!\Delta B^{(N)}\!.
\]

We have $\cov\left(\Delta B^{(N)}, \Delta B^{(N)}\right) = \Gamma^{(N)}$.
Calculate $\cov\left(B_t, \Delta B^{(N)}\right)$:
\[
\cov\left(B_t, \: B_{t_k} - B_{t_{k-1}}\right) =
    \int_{s=0}^t \int_{u=t_{k-1}}^{t_k} K(s-u) \, du \, ds,
\]
and for any vector $x = (x_k)_{k=1}^N \in \mathbb{R}^N$
\begin{align*}
\cov\left(B_t, \Delta B^{(N)}\right) x &= \sum_{k=1}^N \cov\left(B_t, \:  B_{t_k} - B_{t_{k-1}}\right) x_k\\
&= \int_{s=0}^t \sum_{k=1}^N \int_{u=t_{k-1}}^{t_k} K(s-u) x_k \, du \, ds
= \\
& = \int_{0}^t \int_{0}^{T} K(s-u) M^*x(u) \, du \, ds
 = \int_{0}^T \stepzerotot(s) \, \Gamma_T M^* x(s) \, ds = \\
& = \int_{0}^T \Gamma_T \stepzerotot(s) \, M^* x(s) \, ds
 = \left(M \Gamma_T \stepzerotot\right)^\top  x,
\end{align*}
whence we get
\[
    \cov\left(B_t, \Delta B^{(N)}\right) = \left(M \Gamma_T \stepzerotot\right)^\top .
\]
Therefore
\[
\ME [B_t \mid \mathcal{F}_N]
= \left(M \Gamma_T \stepzerotot\right)^\top \bigl(\Gamma^{(N)}\bigr)^{-1} \Delta B^{(N)}
= \int_0^t
\left(\Gamma_T M^* \bigl(\Gamma^{(N)}\bigr)^{-1} \Delta B^{(N)}\right) (s) \, ds .
\]
Then
\begin{align*}
\ME \left[ \int_0^T h_T(t) \, dB_t \Big| \mathcal{F}_N \right]
&= \int_0^T h_T(s) \left(\Gamma_T M^* \bigl(\Gamma^{(N)}\bigr)^{-1} \Delta B^{(N)}\right) (s) \, ds
\\
&= \int_0^T (\Gamma_T h_T)(s)\left (M^* \bigl(\Gamma^{(N)}\bigr)^{-1} \Delta B^{(N)}\right) (s) \, ds
\\
&= (M \Gamma_T h_T)^\top  \bigl(\Gamma^{(N)}\bigr)^{-1} \Delta B^{(N)},
\end{align*}
where we have used that the operator $\Gamma_T$ is self-adjoint.

Further, $M \Gamma_T h_T = M \constone = z$, where the vector $z$ is defined after \eqref{l-def-114}. Hence,
\[
\ME \left[ \int_0^T h_T(t) \, dB_t \Big| \mathcal{F}_N \right]
= z^\top \bigl(\Gamma^{(N)}\bigr)^{-1} \Delta B^{(N)} .
\]

In order to calculate the variance we apply the partition-of-variance
equality
\[
\var \left[ \int_0^T h_T(t) \, dB_t \right] =
\var \left( \ME \left[ \int_0^T h_T(t) \, dB_t \Big| \mathcal{F}_N \right] \right)
+ \var \left[ \int_0^T h_T(t) \, dB_t \Big| \mathcal{F}_N \right].
\]
We have
\[
\var\!\left[ \int_0^T h_T(t) \, dB_t \right] =
\int_0^T (\Gamma_T h_T)(t) \, h_T(t) \, dt
= \int_0^T \constone(t)  h_T(t) \, dt
= \int_0^T h_T(t) \, dt,
\]
and
\[
\var \left( \ME \left[ \int_0^T h_T(t) \, dB_t \Big|
 \mathcal{F}_N \right] \right)
= \var \left( z^\top \bigl(\Gamma^{(N)}\bigr)^{-1} \Delta B^{(N)} \right)
= z^\top \bigl(\Gamma^{(N)}\bigr)^{-1} z .
\]

Hence,
\begin{equation}\label{eq8}
\var \left[ \int_0^T h_T(t) \, dB_t \Big| \mathcal{F}_N \right] =
\int_0^T h_T(t) \, dt - z^\top \bigl(\Gamma^{(N)}\bigr)^{-1} z.
\end{equation}

Applying the formula for the mean of  the log-normal distribution, we obtain
\begin{multline*}
\ME\nolimits_{\theta=0} [ L(\theta) \mid \mathcal{F}_N ]
= \ME \exp \biggl\{
 \theta z^\top \bigl(\Gamma^{(N)}\bigr)^{-1} \Delta B^{(N)}\\
+ \frac{\theta^2}{2} \biggl(
\int_0^T h_T(t) \, dt - z^\top \bigl(\Gamma^{(N)}\bigr)^{-1} z
\biggr)
- \frac{\theta^2}{2} \int_0^T h_T(s) ds \biggr\}
= L^{(N)}(\theta) .
\end{multline*}

Thus, \eqref{eq-152-desired} is proved.
\end{proof}

\begin{corollary}
The maximum likelihood estimator of $\theta$ by continuous observations is given by
\begin{equation}\label{equest}
\hat\theta_T = \frac
{\int_0^T h_T(t) \, dX_t}
{\int_0^T h_T(t) \, dt} .
\end{equation}
\end{corollary}

\subsection{Properties of the estimator}
It follows immediately from \eqref{equest} that the maximum likelihood estimator $\hat\theta_T$
is equal to
\begin{equation}
\label{MLEstcdthce}
\hat\theta_T = \theta + \frac{\int_0^T h_T(t) \, dB_t}
{\int_0^T h_T(t) \, dt} .
\end{equation}

\begin{prop}\label{prop-contt}
The estimator $\hat\theta_T$ is unbiased and
normally distributed.
Its variance is equal to
\begin{equation}
\var \hat\theta_T = \ME\left(\hat\theta_T - \theta\right)^2 =
\frac{1}{\int_0^T h_T(t) \, dt} .
\label{eq-prop-contt-var}
\end{equation}
\end{prop}
\begin{proof}
Unbiasedness and normality follows from the fact that
$\hat\theta - \theta$ is a linear functional of
centered Gaussian process $B$.
By \eqref{eq5a},
\[
\var \biggl( \int_0^T  h_T(t) \, dB_t \biggr) =
\int_0^T \Gamma_T h(t) \, h_T(t) \, dt =
\int_0^T \constone(t) \, h_T(t) \, dt = \int_0^T h_T(t) \, dt.
\]
Thus, equation \eqref{eq-prop-contt-var} immediately follows from
\eqref{MLEstcdthce}.
\end{proof}

\begin{corollary}
Let  the process $B=\{B_t,\; t\ge 0\}$ satisfy
Assumptions 1 and 2.
If
\begin{equation}\label{eq:denom}
  \int_0^T h_T(t) \, dt \to \infty,
  \quad \textrm{as} \quad T\to +\infty,
\end{equation}
then the maximum likelihood estimator $\hat\theta_T$
is mean-square consistent, i.e.,
$  \ME( \hat\theta_T - \theta)^2 \to 0$,
as $T\to +\infty$.
\end{corollary}

It can be hard to verify the condition \eqref{eq:denom}.
The following result gives sufficient conditions for the consistency in terms of the autocovariance function of $B$.
\begin{theorem}\label{thm-conscontt}
Let  the process $B=\{B_t,\; t\ge 0\}$ satisfy
Assumptions 1 and 2.
If the covariance function of increment process $B_N - B_{N-1}$
tends to 0:
\[
 \ME \left(B_{ N+1 } - B_{N}\right) B_1 \to 0 \quad
 \textrm{as} \quad
 N\to\infty,
\]
then the maximum likelihood estimator $\hat\theta_T$
is mean-square consistent.
\end{theorem}

\begin{proof}
The estimator $\hat\theta^{(N)}$ from the discrete sample $\{X_1,\ldots,X_N\}$
is mean-square consistent by Theorem~\ref{thm-consdiskr}.
The estimator from the continuous-time sample\linebreak $\{X_t,\; t\in[0,T]\}$
is unbiased. Now compare the variances of the discrete and continuous-time
estimators.

The desired inequalities are got from the proof of
Theorem~\ref{thm-consdiskr}.
Suppose that $T\ge 1$, $N$ is an integer such that $N \le T < N+1$.
By equation \eqref{eq8}
we have
\begin{equation}
\label{comp-Tinteger}
\int_0^T h_T(t) \, dt \ge z^\top \bigl(\Gamma^{(N)}\bigr)^{-1} z.
\end{equation}
As $\var \hat\theta_T = \frac{1}{\int_0^T h_T(t) \, dt}$,
$\var \hat\theta^{(N)} = \left(z^\top \bigl(\Gamma^{(N)}\bigr)^{-1} z\right)^{-1}$,
we have
$\var\hat\theta_T \le \var \hat\theta^{(N)}$, and
\[
\lim_{T\to+\infty} \ME\left(\hat\theta_T - \theta\right)^2
=\lim_{N\to\infty} \ME\left(\hat\theta^{(N)} - \theta\right)^2 = 0.
\qedhere
\]
\end{proof}

To prove the strong consistency of $\hat\theta_T$, we need the following auxiliary result.
\begin{lemma}
Let the process $B$ satisfy the conditions of Theorem~\ref{th-L}. Then
the estimator process   $\hat\theta=\{\hat\theta_T, T\geq 0\}$ has independent increments.
\end{lemma}

\begin{proof}
Let  $T_2 \le T_3$. Then
\begin{align*}
\ME \left[\int_0^{T_3} h_{T_3}(t) \, dB_t
          \int_0^{T_2} h_{T_2}(s) \, dB_s \right]
& = \int_0^{T_2} \Gamma_{T_3} h_{T_3}(t) \, h_{T_2}(t) \, dt
\\ &= \int_0^{T_2} \indicator\nolimits_{[0,T_3]}(t) \, h_{T_2}(t) \, dt
    = \int_0^{T_2} \, h_{T_2}(t) \, dt .
\end{align*}
Thus, if $0 < T_1 \le T_2 \le T_3 \le T_4$, then
\begin{multline*}
\ME (\hat\theta_{T_4} - \hat\theta_{T_3})
    (\hat\theta_{T_2} - \hat\theta_{T_1}) \\
\begin{aligned}
&=
\newcommand{\hatthetaTk}[1]{
\frac{\int_0^{T_#1} h_{T_#1}(t) \, dB_t}
     {\int_0^{T_#1} h_{T_#1}(t) \, dt}}
\ME \left( \hatthetaTk{4} - \hatthetaTk{3} \right)
    \left( \hatthetaTk{2} - \hatthetaTk{1} \right)
= \\ & =
\newcommand{\covhatthetaTkTk}[2]{
\frac{\ME \left[\int_0^{T_#1} h_{T_#1}(t) \, dB_t
                \int_0^{T_#2} h_{T_#2}(t) \, dB_t\right]}
{\int_0^{T_#1} h_{T_#1}(t) \, dt
 \int_0^{T_#2} h_{T_#2}(t) \, dt}}
\covhatthetaTkTk{4}{2} - \covhatthetaTkTk{3}{2}
- \\ &\quad -
\newcommand{\covhatthetaTkTk}[2]{
\frac{\ME \left[\int_0^{T_#1} h_{T_#1}(t) \, dB_t
                \int_0^{T_#2} h_{T_#2}(t) \, dB_t\right]}
{\int_0^{T_#1} h_{T_#1}(t) \, dt
 \int_0^{T_#2} h_{T_#2}(t) \, dt}}
\covhatthetaTkTk{4}{1} + \covhatthetaTkTk{3}{1}
= \\ & =
\newcommand{\fracdenom}[1]{
\frac{1}{\int_0^{T_#1} h_{T_#1}(t) \, dt}}
\fracdenom{4} - \fracdenom{3} - \fracdenom{4} + \fracdenom{3} = 0 .
\end{aligned}
\end{multline*}
Similarly to the proof of Lemma \ref{lem-indincrdiscr}, the random process $\hat\theta_T$ is Gaussian and  its increments
are proved to be uncorrelated so they  are independent.
\end{proof}

\begin{theorem}
Under conditions of Theorem~\ref{thm-conscontt}
the estimator $\hat\theta_T$ is strongly consistent.
\end{theorem}
\begin{proof}
By Kolmogorov's inequality, for any $\epsilon>0$ and $t_0 > 0$
\begin{align*}
\Prob\left( \sup_{T>t_0} |\hat\theta_T - \theta| > \epsilon \right)
&\le
\Prob\left( |\hat\theta_{t_0} - \theta| > \frac{\epsilon}{2} \right) +
\Prob\left( \sup_{T>t_0} |\hat\theta_T - \theta_{t_0}| > \frac{\epsilon}{2}
\right)
\le \\ & \le
\frac{4}{\epsilon^2} \var \hat\theta_{t_0} +
\frac{4}{\epsilon^2} \lim_{T\to+\infty}
\var\left(\hat\theta_T - \hat\theta_{t_0}\right) =
\frac{8}{\epsilon^2} \var \hat\theta_{t_0} .
\end{align*}
By Theorem~\ref{thm-conscontt},
\[
\lim_{t_0 \to +\infty} \Prob\left( \sup_{T>t_0} |\hat\theta_T - \theta| > \epsilon \right) = 0
\quad \mbox{for all} \quad \epsilon>0,
\]
whence the strong consistency follows.
\end{proof}

\begin{remark}
The Brownian motion does not satisfy Assumption~\ref{assump-B}
(as for covariance function $\max(s,t)$ of Wiener process,
$\frac{\partial \max(s,t)}{\partial t}$ is not continuous in $s$).
So we extend our model such that it can handle Wiener process.
Let the process $B$ be a sum of two independent random processes,
\begin{equation}\label{eq-Bcomb}
B_t = B_t^{\rm C} + W_t,
\end{equation}
where $B^{\rm C}$ satisfy Assumption~\ref{assump-B}, and
$W$ is a standard Wiener process.
Let us look at the changes of the statements if the process $B$ admits
representation \eqref{eq-Bcomb}
(Assumption~\ref{assump-B} for $B$ is dropped).
Lemma~\ref{lem-covfun41} changes as follows:
\[
\ME \left[ \int_0^T f(t) \, dB_t \int_0^T g(s) \, dB_s\right]
= \int_0^T f(t) \int_0^T K(t-s) g(s)\, ds \, dt
+ \int_0^T f(t) g(t) \, dt .
\]
Equation \eqref{eq5a} will stand true, if we set
\[
\Gamma_T f(t) = f(t) + \Gamma_T^{\rm C} f(t) =
f(t) + \int_0^T K(t-s) f(s) \, ds .
\]
Lemma~\ref{lem-operatorGamma} stands true
(with $\|\Gamma_T\| \le \|K\|_{L_1[-T,T]} + 1$).
Theorem~\ref{th-L} holds true; minor changes in the proof are required.
\end{remark}

\subsection{Examples}
\begin{example}\label{ex-Bh}
Let $B$ be a fractional Brownian motion with the Hurst index $H\in(1/2, 1)$.
Then $K(t) = \frac{H (2H-1)}{|t|^{2-2H}}$. We denote by $\Gamma^H_T$ the corresponding operator $\Gamma_T$.
Then for the function
\[
h_T(s) = C_H
s^{1/2-H} (T-s)^{1/2-H},
\]
$C_H=\left(H (2H-1) \mathrm{B} \left(H-\frac12, \frac32-H\right)\right)^{-1}$,
we have that
$\Gamma^H_T h_T = \constone$,
see~\cite{Norros1999}.
The maximum likelihood estimator is given by
\[
\hat\theta_T = \frac{T^{2H - 2}}{\Betafunction(3/2-H,\: 3/2-H)}
\int_0^T s^{1/2-H} (T-s)^{1/2-H} \, dX_s.
\]
\end{example}

\begin{example}\label{ex:B+W}
Consider the following model:
\begin{equation}\label{eq:mixed}
X_t = \theta t + W_t + B_t^{H},
\end{equation}
where $W$ is a standard Wiener process,
$B^{H}$ is a fractional Brownian motion
with Hurst index $H$,
and random processes $W_t$ and $B_t^{H}$ are independent.
The process $W_t + B_t^{H}$ admits representation
\eqref{eq-Bcomb} with $B^{\rm C}=B^H$.
Corresponding operator $\Gamma_T$ is
$\Gamma_T = I + \Gamma^H_T$
(see Example~\ref{ex-Bh} for the definition of $\Gamma^H_T$).
The operator $\Gamma^H_T$ is self-adjoint and positive semi-definite.
Hence, the operator $\Gamma_T$ is invertible.  Thus
Assumption~\ref{assump-hexists} holds true.

The function $h_T = \Gamma_T^{-1} \constone$ can be evaluated iteratively
\begin{equation}\label{eq:approx}
h_T = \sum_{k=0}^\infty
\frac{\left( \frac12 \, \norm{\Gamma^H_T} \, I - \Gamma^H_T\right)^k \constone}
     {\left( 1 + \frac12 \, \norm{\Gamma^H_T}\right)^{k+1}} \,.
\end{equation}
\end{example}

\subsection{Simulations}
We illustrate the behavior of the maximum likelihood estimator for Example~\ref{ex:B+W} with the help of simulation experiments.
For $T=1$ and $T=10$ and various values of $H$ we find $h_T$ iteratively by \eqref{eq:approx}.
Then for $\theta=2$ we simulate 1000 realizations of the process~\eqref{eq:mixed} for each $H$ and
compute the estimates by \eqref{equest}.
The means and variances of these estimates are reported in Table~\ref{tab:1}.
The theoretical variances calculated by \eqref{eq-prop-contt-var} are also presented.
We see that these simulation studies confirm the theoretical properties of $\widehat\theta_T$, especially unbiasedness and consistency.

\begin{table}
\caption{The means and variances of $\hat\theta_T$}
\centering
\footnotesize
\begin{tabular}{llccc}\toprule
\multirow{2}*{$H$} & \multirow{2}*{$T$} & Sample & Sample & Theoretical\\
& & Mean & Variance & Variance\\
\otoprule
\multirow{2}*{$0.6$}
& $1$   & $2.0344$ & $1.9829$ & $1.8292$\\
& $10$ & $2.0047$ & $0.2584$ & $0.2356$\\
\addlinespace
\multirow{2}*{$0.7$}
& $1$   & $1.9995$ & $1.9956$ & $1.9692$\\
& $10$ & $2.0296$ & $0.3484$ & $0.3270$\\
\addlinespace
\multirow{2}*{$0.8$}
& $1$   & $2.0165$ & $2.0435$ & $1.9930$\\
& $10$ & $2.0071$ & $0.5155$ & $0.4392$\\
\addlinespace
\multirow{2}*{$0.9$}
& $1$   & $2.0117$ & $2.0376$ & $1.9984$\\
& $10$ & $2.0099$ & $0.7710$ & $0.5867$\\
\bottomrule
\end{tabular}
\label{tab:1}
\end{table}

\def\ocirc#1{\ifmmode\setbox0=\hbox{$#1$}\dimen0=\ht0 \advance\dimen0
  by1pt\rlap{\hbox to\wd0{\hss\raise\dimen0
  \hbox{\hskip.2em$\scriptscriptstyle\circ$}\hss}}#1\else {\accent"17 #1}\fi}

\end{document}